\documentclass[11pt,reqno,a4paper]{amsart}

\usepackage{amsmath,amsthm,amssymb,mathtools,bbm}
\usepackage{esint}
\usepackage{bookmark}
\usepackage{hyperref}
\hypersetup{pdfstartview={FitH}}

\theoremstyle{plain}
\newtheorem{theorem}{Theorem}
\newtheorem{proposition}[theorem]{Proposition}
\newtheorem{lemma}[theorem]{Lemma}

\newtheorem*{problem}{Problem}

\numberwithin{equation}{section}

\renewcommand{\le}{\leqslant}

\renewcommand{\ge}{\geqslant}

\newcommand{\Z}{\mathbb{Z}}
\newcommand{\R}{\mathbb{R}}
\newcommand{\N}{\mathbb{N}}
\newcommand{\Q}{\mathbb{Q}}
\newcommand{\calB}{\mathcal{B}}
\newcommand{\Ball}[2]{\textup{B}_{#1}(#2)}
\newcommand{\dd}{\,\textup{d}}
\newcommand{\Dil}{\textup{Dil}}
\newcommand{\PP}{\textup{P}}

\begin{document}

\title{A Banach space that distinguishes two maximal operators}

\author[Vjekoslav Kova\v{c}]{Vjekoslav Kova\v{c}}
\address{University of Zagreb Faculty of Science, Department of Mathematics, Bijeni\v{c}ka cesta 30, 10000 Zagreb, Croatia}
\email{vjekovac@math.hr}



\begin{abstract}
Maz'ya and Shaposhnikova introduced a non-classical maximal operator $\textup{M}^{\diamond}$ as the maximal convolution with the vector-valued signum kernel truncated to centered balls. We construct a translation-invariant Banach space of locally integrable functions on which $\textup{M}^{\diamond}$ is bounded, but the sharp maximal operator $\textup{M}^{\sharp}$ is not. This answers one of Maz'ya's questions from a collection of $75$ open problems in analysis.
\end{abstract}

\maketitle



\section{Introduction}

\subsection{The problem and its background}
Maz'ya and Shaposhnikova \cite{MazyaShaposhnikova2002} introduced the following maximal operator appearing naturally in the context of certain sharp pointwise interpolation inequalities for derivatives of higher-dimensional functions:
\begin{equation}\label{eq:Mdiamond-def}
(\textup{M}^{\diamond}f)(x) :=
\sup_{r>0} \Big\lvert \fint_{\Ball{r}{x}}\frac{y-x}{|y-x|}f(y)\dd y \Big\rvert, \quad x\in\R^n.
\end{equation}
It is the maximal convolution with dilates of the vector-valued kernel
\begin{equation}\label{eq:Mdiamond-kernel}
K^{\diamond}(x) :=
\begin{cases}
\displaystyle\frac{1}{|\Ball{1}{0}|}\frac{x}{|x|} & \text{if } 0<|x|\le1,\\
0 & \text{otherwise}.
\end{cases}
\end{equation}
Maz'ya also discussed \eqref{eq:Mdiamond-def} at the very beginning of Chapter 12 of the now classical book on Sobolev spaces \cite{Mazya2011}. More recently, in the published problem list \cite{Mazya2018}, Maz'ya asked about the following distinction between $\textup{M}^{\diamond}$ and the classical Fefferman--Stein sharp maximal operator:
\begin{equation}\label{eq:Msharp-def}
(\textup{M}^{\sharp}f)(x) :=
\sup_{r>0} \fint_{\Ball{r}{x}}\Big\lvert f(y)-\fint_{\Ball{r}{x}}f \Big\rvert \dd y, \quad x\in\R^n.
\end{equation}

\begin{problem}[{\cite[Section~4.7, Problem~20]{Mazya2018}}]
Does there exist a Banach space $B$ such that one of the operators $\textup{M}^{\sharp}$ or $\textup{M}^{\diamond}$ is bounded in $B$ whereas the other operator is not bounded?
\end{problem}

It is implicitly understood that $B$ should be contained in the vector space of a.e.-equivalence classes of locally integrable functions, $\textup{L}^{1}_{\mathrm{loc}}(\R^{n})$, since otherwise $\textup{M}^{\diamond}$ and $\textup{M}^{\sharp}$ would not be meaningfully defined on all of $B$.
We might additionally aim to construct a translation-invariant space $B$, since most important spaces in the literature are, in fact, invariant under translations $f\mapsto f(\cdot-x_0)$.

We are not aware of previous work addressing this boundedness distinction question and the present note gives an affirmative answer to it.
By contrast, there are several broad classes of spaces on which the two operators in Maz'ya's question are both bounded. If we also write
\[ (\textup{M}f)(x):=\sup_{r>0}\fint_{\Ball{r}{x}}|f|, \quad x\in\R^n \]
for the Hardy--Littlewood maximal operator, then the elementary inequalities
\begin{equation}\label{eq:domination-by-HL}
\textup{M}^{\diamond}f\le \textup{M}^{\sharp}f\le 2\textup{M}f
\end{equation}
hold pointwise for every locally integrable $f$, the first one being an immediate consequence of the fact that $\textup{M}^{\diamond}$ was defined using the cancellative kernel \eqref{eq:Mdiamond-kernel}. Consequently, if $X$ is a Banach function lattice on which $\textup{M}$ is bounded, then both $\textup{M}^{\diamond}$ and $\textup{M}^{\sharp}$ are bounded on $X$ as well.
That way the common boundedness of $\textup{M}^{\diamond}$ and $\textup{M}^{\sharp}$ is already known on all the standard spaces for which the Hardy--Littlewood maximal theorem is available: for example, on $\textup{L}^{p}(\R^{n})$ for $1<p\le\infty$ \cite{Stein1970}, on the weighted spaces $\textup{L}^{p}(w)$ for $1<p<\infty$ and $w\in A_{p}$ \cite{Muckenhoupt1972,GarciaCuervaRubio1985}, on rearrangement-invariant Banach function spaces whose lower Boyd index is larger than one, in particular on Lorentz spaces $\textup{L}^{p,q}(\R^n)$ with $1<p<\infty$ and $1\le q\le\infty$ \cite{BennettSharpley1988,LoristNieraeth2024}, and on variable Lebesgue spaces $\textup{L}^{p(\cdot)}(\R^n)$ under the usual hypotheses guaranteeing boundedness of $\textup{M}$, 
such as local log-H\"{o}lder continuity together with log-H\"{o}lder decay at infinity, and $1<p_-\le p_+<\infty$ \cite{Diening2004,CruzUribeFiorenza2013}.

The problem is closely related to, but apparently distinct from, several well-developed lines of work.
The first related body of work concerns the Fefferman--Stein sharp function and its good $\lambda$ inequalities. The classical theory gives estimates of $\textup{M}f$ in terms of $\textup{M}^{\sharp}f$ and is central in weighted inequalities and singular integrals \cite{FeffermanStein1971,FeffermanStein1972,CoifmanFefferman1974,GarciaCuervaRubio1985}. A closely related comparison problem for variants of the sharp maximal function was studied by Martell, who introduced sharp maximal functions associated with approximations of the identity on spaces of homogeneous type, proved good $\lambda$ estimates for them, and noted that they may be pointwise smaller than the classical sharp maximal function while not being comparable to it in general \cite{Martell2004}. Lerner later characterized Fefferman--Stein inequalities on Banach function spaces in terms of boundedness properties of $\textup{M}$ \cite{Lerner2010}, and related Fefferman--Stein equivalences and maximal operator criteria have continued to be studied \cite{Lerner2020,Nieraeth2023,Lerner2025}.

A second related line of work studies maximal operators at the endpoint $p=\infty$, especially on $\textup{BMO}$; see \cite{CoifmanRochberg1980,BennettDeVoreSharpley1981,Bennett1982,GarciaCuervaMaciasTorrea1998,Lerner2007}. However, in our case, boundedness statements on $\textup{BMO}$ are immediate. With the present convention of using balls in the definitions,
\[ \|\textup{M}^{\sharp}f\|_{\textup{L}^{\infty}(\R^n)} = \|f\|_{\textup{BMO}(\R^n)}. \]
Hence
\[ \|\textup{M}^{\sharp}f\|_{\textup{BMO}(\R^n)}\lesssim \|f\|_{\textup{BMO}(\R^n)} \]
and, because of \eqref{eq:domination-by-HL}, the same conclusion holds for $\textup{M}^{\diamond}$.

This places us in the third context, the study of smoothness/regularity properties of maximal operators, where the norm is not merely a lattice norm. Kinnunen proved the boundedness of the Hardy--Littlewood maximal operator on $W^{1,p}(\R^{n})$, $1<p\le\infty$ \cite{Kinnunen1997}. This initiated substantial regularity theory for maximal operators on other Sobolev spaces \cite{KinnunenSaksman2003,HajlaszOnninen2004,Luiro2007,AaltoKinnunen2008,BRS19,LM19,BM20,BM21,KinnunenLehrbackVahakangas2021,CGM22,BGMW23}, H\"{o}lder spaces \cite{Buckley1999}, Triebel--Lizorkin spaces \cite{Korry2002}, and functions of bounded variation \cite{CS13,CM17,CMP17,CFS18,Madrid19}.
However, many of these spaces do not fully lie in $\textup{L}^{1}_{\mathrm{loc}}(\R^{n})$ for some values of their parameters, which invalidates them as potential examples. The author was not able to find a single classical space among them that actually solves the problem.

\subsection{The main result}
The purpose of this note is to prove that the two operators \eqref{eq:Mdiamond-def} and \eqref{eq:Msharp-def} nevertheless have different boundedness properties. 
The heuristic for the construction relies on the fact that \eqref{eq:Mdiamond-def} essentially takes one Littlewood--Paley piece at a time of the input function. For instance, in one dimension, the Fourier transform of \eqref{eq:Mdiamond-kernel} is $(1-\cos(2\pi\xi))/(2\pi\mathbbm{i}\xi)$, which decays, albeit slowly. On the other hand, \eqref{eq:Msharp-def} always takes the full frequency portrait of the input function. Thus, the two maximal operators should be distinguished by a Besov norm composed as an $\ell^\infty$ quantity of the sizes of the Littlewood--Paley pieces. It will ignore the accumulation of many separated frequencies; $\textup{M}^{\diamond}$ respects that structure, while $\textup{M}^{\sharp}$ destroys it by measuring absolute oscillation. It is then easy to construct a sequence of functions $(f_N)_{N=1}^\infty$ such that
\[ \frac{\|\textup{M}^{\sharp}f_N\|}{\|f_N\|} \gtrsim N^{1/2} \rightarrow\infty \]
in that norm.
The only problem is that such a Banach space is not contained in $\textup{L}^{1}_{\mathrm{loc}}(\R^{n})$, but is a space of distributions. We patch this by intersecting it, somewhat artificially, with the space of square-integrable functions. The Besov part and the square-integrable part scale differently under dilations $x\mapsto \lambda x$ of the domain $\R^n$.
Despite the fact that both $\textup{M}^{\diamond}$ and $\textup{M}^{\sharp}$ are bounded on $\textup{L}^{2}(\R^{n})$, we can thus scale the functions $f_N$ carefully to make the $\textup{L}^2$ part harmless and, in the new norm, preserve the above property with only a logarithmic loss:
\[ \frac{\|\textup{M}^{\sharp}f_N\|}{\|f_N\|} \gtrsim \frac{N^{1/2}}{\log N} \rightarrow\infty. \]

We turn to the details of the construction.
Fix a decreasing function $\varphi\in\textup{C}^\infty(\R)$ that is identically $1$ on $(-\infty,1]$ and identically $0$ on $[2,\infty)$. Define $n$-dimensional Schwartz functions $h_{\le0}$, $h_1$, $h_2$, $\ldots$ via their Fourier transforms as
\begin{align*}
\widehat{h_{\le0}}(\xi) & := \varphi(|\xi|), \quad \xi\in\R^n,\\
\widehat{h_j}(\xi) & := \varphi\Bigl(\frac{|\xi|}{2^j}\Bigr) - \varphi\Bigl(\frac{|\xi|}{2^{j-1}}\Bigr), \quad \xi\in\R^n, \quad j\ge1.
\end{align*}
Convolution-type operators $\PP_{\!\le0},\PP_1,\PP_2,\ldots$ defined as
\begin{align*}
\PP_{\!\le0}f & := f \ast h_{\le0},\\
\PP_j f & := f \ast h_j, \quad j\ge1,
\end{align*}
then constitute a smooth inhomogeneous Littlewood--Paley decomposition:
\[ \PP_{\!\le0}+ \sum_{j=1}^{\infty}\PP_{j} = I. \]
In words, $\PP_{\!\le0}$ is a smooth projection to low frequencies $|\xi|\le 2$, while $\PP_{j}$ are smooth projections to frequencies $2^{j-1}\le |\xi|\le 2^{j+1}$.

For $f\in\textup{L}^{2}(\R^n)$ we introduce a norm which is the sum of a Besov-type norm and the ordinary $\textup{L}^2$ norm,
\[ \|f\|_{B} := \|\PP_{\!\le0}f\|_{\textup{L}^{\infty}(\R^n)} + \sup_{j\ge1} \|\PP_{j}f\|_{\textup{L}^{\infty}(\R^n)} + \|f\|_{\textup{L}^2(\R^n)}, \]
and define the corresponding function space,
\[ B := \bigl\{ f\in\textup{L}^{2}(\R^n) \,:\, \|f\|_{B}<\infty \bigr\}. \]
It is clear that $B$ is translation-invariant and
\[ B \subseteq \textup{L}^{2}(\R^n)\subseteq \textup{L}^{1}_{\mathrm{loc}}(\R^{n}). \]
It is also easy to see that $B$ is complete, simply because $\textup{L}^2(\R^n)$, $\textup{L}^\infty(\R^n)$, and the vector-valued space $\ell^\infty\bigl(\mathbb{N};\textup{L}^\infty(\mathbb{R}^n)\bigr)$ are all complete. Note that $B$ is not a Banach lattice.

Here is our main result.

\begin{theorem}\label{thm:main}
In every dimension $n\ge1$, the maximal operator $\textup{M}^{\diamond}$ is bounded on $B$ while $\textup{M}^{\sharp}$ is not.
\end{theorem}

In fact, $\textup{M}^{\sharp}$ does not even map $B$ into itself: there exists $f\in B$ such that $\textup{M}^{\sharp}f\notin B$. We prove this stronger assertion in Proposition~\ref{prop:nonself} below.

Our proof of Theorem~\ref{thm:main} is self-contained, apart from standard textbook facts from harmonic analysis.
In Section~\ref{sec:positive}, the boundedness of $\textup{M}^{\diamond}$ will be obtained as a special case of an endpoint estimate for more general cancellative dilation kernels, namely, compactly supported bounded mean-zero kernels whose components have bounded variation.
In Section~\ref{sec:negative} the failure of boundedness of $\textup{M}^{\sharp}$ on $B$ will be witnessed by an example constructed from lacunary trigonometric series.

While working on the problem, the author consulted OpenAI's \emph{ChatGPT} 5.4 Pro, which helped rule out (as potential examples) the classical function spaces covered in the literature. After the author finalized the proof, a new model, \emph{ChatGPT} 5.5 Pro, was released. It produced the following surprisingly simple but not entirely legitimate alternative example.

Let $\widetilde{B}$ be the space of all functions $f:\R^n\to\mathbb{C}$ that are Lipschitz, bounded, radial, and satisfy $f(0)=0$. Its appropriate norm is
\[ \|f\|_{\widetilde{B}} :=
\sup_{\substack{x,y\in\R^n\\ x\neq y}}\frac{|f(x)-f(y)|}{|x-y|}
+ \sup_{x\in\R^n}|f(x)|. \]
It is an exercise to show that $\textup{M}^{\diamond}$ is bounded on $\widetilde{B}$. However, $\textup{M}^{\sharp}$ does not even map $\widetilde{B}$ to itself, as is seen by considering the function $f(x):=\min\{|x|,1\}$.
Note that this space is clearly not translation-invariant, which, as we have already commented, might have been an implicit assumption behind Maz'ya's problem.

Two further questions are not addressed in this paper.
\begin{quote}
\emph{Does there exist a translation-invariant Banach function lattice $X\subseteq\textup{L}^1_{\mathrm{loc}}(\R^n)$ on which $\textup{M}^{\diamond}$ is bounded but $\textup{M}^{\sharp}$ is not?}
\end{quote}
\begin{quote}
\emph{Does there exist a translation-invariant Banach space $X\subseteq\textup{L}^1_{\mathrm{loc}}(\R^n)$ on which $\textup{M}^{\sharp}$ is bounded but $\textup{M}^{\diamond}$ is not?}
\end{quote}
A space answering the second question cannot be a Banach function
lattice, since we have observed that \eqref{eq:domination-by-HL} holds.
The author is grateful to the anonymous referee for a suggestion to include these questions.

\subsection{Notation}
For two nonnegative expressions $A$ and $B$, we write $A\lesssim B$ and $B\gtrsim A$ if $A\le C B$ holds with some unimportant (finite) constant $C\ge0$. If we want to emphasize that $C$ also depends on some set of parameters $P$, we write instead $A\lesssim_P B$ and $B\gtrsim_P A$. Most such dependencies, such as on the dimension $n$ of the ambient space, will be understood implicitly.
We write $A\sim B$ when both $A\lesssim B$ and $B\lesssim A$ hold.

The \emph{Fourier transform}, initially defined for integrable functions $f\colon\R^n\to\mathbb{C}$, will be written and normalized as
\[ \widehat{f}(\xi) := \int_{\R^n} f(x) e^{-2\pi \mathbbm{i} x\cdot\xi} \dd x. \]
The \emph{convolution} of $f$ and $g$ is defined by
\[ (f\ast g)(x) := \int_{\R^n} f(x-y) g(y) \dd y, \]
as usual.


\section{Boundedness of the Maz'ya--Shaposhnikova maximal operator}
\label{sec:positive}
For the main part of the argument we can work with a more general maximal operator than $\textup{M}^{\diamond}$. Take a positive integer $m$. Let
\[ K=(K^1,\ldots,K^m):\R^n\to\R^m \]
be a compactly supported vector-valued kernel such that each of its components belongs to $\textup{L}^\infty(\R^n)\cap \textup{BV}(\R^n)$ and
\begin{equation}\label{eq:general_ker}
\int_{\R^n}K^l(x)\dd x=0
\end{equation}
for $l=1,2,\ldots,m$.
It is clear that the kernel $K^{\diamond}$ from \eqref{eq:Mdiamond-kernel} satisfies all these requirements.

Namely, recall that the distributional derivative 
\[ \textup{D}K^l=(\partial_1K^l,\ldots,\partial_nK^l) \]
is a finite $\R^n$-valued Radon measure for each $l$. We use the following representation of its total variation:
\[ \|\textup{D}K^l\|_{\textup{TV}} := |\textup{D}K^l|(\R^n)
= \sup_{\substack{\Phi\in\textup{C}_c^1(\R^n;\R^n)\\ \sup_{x\in\R^n}|\Phi(x)|\le1}}
\left|\int_{\R^n}K^l(x)\operatorname{div}\Phi(x)\dd x\right|. \]
This is the $\textup{BV}$ seminorm, while the full norm on the space of functions of bounded variation $\textup{BV}(\R^n)$ is
\[ \|K^{l}\|_{\textup{BV}} := \|K^l\|_{\textup{L}^1}+\|\textup{D}K^l\|_{\textup{TV}}. \]

For $r>0$, denote the $\textup{L}^1$-dilates of $K$ by
\[ K_r(x)=\frac{1}{r^{n}}K\Bigl(\frac{x}{r}\Bigr) \]
and define the abstract maximal function as
\begin{equation}\label{eq:TK_star2}
(T^\star_K f)(x)=\sup_{\Q\ni r>0}|(K_r\ast f)(x)|.
\end{equation}
Since $K$ is essentially bounded and compactly supported, the integral defining $(K_r\ast f)(x)$ is absolutely convergent whenever $f\in\textup{L}^1_{\mathrm{loc}}(\R^n)$.
For $f\in\textup{L}^2(\R^n)$, the map $r\mapsto(K_r\ast f)(x)$ is continuous for every $x$, since $K\in\textup{L}^2(\R^n)$ and dilations are continuous in $\textup{L}^2$. Thus the supremum over positive rational radii equals the supremum over all positive radii.

In the proof below we will use the standard translation estimate 
\begin{equation}\label{eq:Kltran}
\|K_r^l(\cdot-y)-K_r^l\|_{\textup{L}^1(\R^n)} \le |y|\,\|\textup{D}K_r^l\|_{\textup{TV}},
\end{equation}
and the dilation identity
\begin{equation}\label{eq:Kldil}
\|\textup{D}K_r^l\|_{\textup{TV}} = r^{-1} \|\textup{D}K^l\|_{\textup{TV}}. 
\end{equation}

Functions $h_1,h_2,\ldots$ from the construction of the Littlewood--Paley decomposition clearly satisfy
\begin{equation}\label{eq:standardLP}
\int_{\R^n} h_j=0,\ \|h_j\|_{\textup{L}^1(\R^n)}\lesssim1,\ \|\nabla h_j\|_{\textup{L}^1(\R^n;\R^n)}\lesssim2^j,\ \int_{\R^n}|x|\,|h_j(x)|\dd x\lesssim2^{-j}
\end{equation}
for every $j\ge1$.
Also define slightly enlarged smooth cutoffs $\widetilde{h}_1,\widetilde{h}_2,\ldots$ via
\[ \widehat{\widetilde{h}_j}(\xi) := \varphi\Bigl(\frac{|\xi|}{2^{j+1}}\Bigr) - \varphi\Bigl(\frac{|\xi|}{2^{j-2}}\Bigr), \quad \xi\in\R^n, \quad j\ge1 \]
and observe that they satisfy the same type of estimates:
\begin{equation}\label{eq:htildeLP}
\int_{\R^n} \widetilde{h}_{j}=0, \ \|\widetilde{h}_j\|_{\textup{L}^1(\R^n)}\lesssim1, \ \|\nabla \widetilde{h}_{j}\|_{\textup{L}^{1}(\R^n;\R^n)}\lesssim 2^{j}, \ \int_{\R^{n}} |x|\,|\widetilde{h}_{j}(x)|\dd x\lesssim 2^{-j}.
\end{equation}
The auxiliary projections
\[ \widetilde{\PP}_j f := f \ast \widetilde{h}_j, \quad j\ge1 \]
then clearly satisfy
\begin{equation}\label{eq:DDisD}
\widetilde{\PP}_{j}\PP_{j}=\PP_{j}.
\end{equation}

All implicit constants in the notation $\lesssim$ below are allowed to depend on $m$, $n$, $K$, and the functions $h_j$ and $\widetilde{h}_j$.

\begin{lemma}
For $r>0$ and $j\ge 1$,
\begin{equation}\label{eq:general-kernel-estimate}
\|K_{r}\ast \widetilde{h}_{j}\|_{\textup{L}^{1}(\R^n;\R^m)} \lesssim_K \min\{2^{j}r,(2^{j}r)^{-1}\}.
\end{equation}
\end{lemma}

\begin{proof}
For the estimate by $2^j r$, we use the cancellation of $K_r$,
\begin{align*} 
(K_r\ast \widetilde h_j)(x)
& \stackrel{\eqref{eq:general_ker}}{=} \int_{\R^n} K_r(y)\bigl(\,\widetilde h_j(x-y)-\widetilde h_j(x)\bigr)\dd y \\
& \ = - \int_{\R^n} \int_0^1 K_r(y) \bigl(y\cdot\nabla\widetilde{h}_j(x-\theta y)\bigr) \dd\theta \dd y, 
\end{align*}
followed by Minkowski's inequality for integrals:
\begin{align*}  
\|K_r\ast \widetilde{h}_j\|_{\textup{L}^{1}(\R^n;\R^m)}
& \le \int_{\R^n} |y| \,|K_r(y)|\, \|\nabla\widetilde{h}_j\|_{\textup{L}^1(\R^n;\R^n)}\dd y \\
& \!\!\stackrel{\eqref{eq:htildeLP}}{\lesssim} 2^j r \int_{\R^n} |y| \,|K(y)| \dd y
\lesssim 2^j r.
\end{align*}

For the complementary estimate, use the cancellation of $\widetilde h_j$,
\[ (K_r\ast \widetilde h_j)(x)
\stackrel{\eqref{eq:htildeLP}}{=} \int_{\R^n} \widetilde h_j(y)\bigl(K_r(x-y)-K_r(x)\bigr)\dd y, \]
and apply Minkowski's inequality again:
\begin{align*}
\|K_r\ast \widetilde h_j\|_{\textup{L}^{1}(\R^n;\R^m)}
& \le \int_{\R^n} |\widetilde h_j(y)|\,\|K_r(\cdot-y)-K_r\|_{\textup{L}^{1}(\R^n;\R^m)} \dd y \\
& \stackrel{\eqref{eq:Kltran}}{\le} \int_{\R^n} |\widetilde h_j(y)|\,|y| \Bigl(\sum_{l=1}^{m}\|\textup{D}K^l_r\|_{\textup{TV}}\Bigr) \dd y \\
&  \stackrel{\eqref{eq:Kldil}}{=} r^{-1}  \Bigl(\sum_{l=1}^{m}\|\textup{D}K^l\|_{\textup{TV}}\Bigr) \int_{\R^n} |y|\,|\widetilde h_j(y)| \dd y
\stackrel{\eqref{eq:htildeLP}}{\lesssim} 2^{-j} r^{-1}.
\end{align*}
This proves the lemma.
\end{proof}

Denote, for brevity, the Besov part of the norm $\|\cdot\|_B$ by
\[ \|f\|_{\calB} := \|\PP_{\!\le0}f\|_{\textup{L}^{\infty}(\R^n)} + \sup_{j\ge1} \|\PP_{j}f\|_{\textup{L}^{\infty}(\R^n)}. \]
Let us record an elementary embedding,
\begin{equation}\label{eq:eemb}
\|f\|_{\calB} \lesssim \|f\|_{\textup{L}^{\infty}(\R^n)},
\end{equation}
which is an immediate consequence of the uniform $\textup{L}^1$ bounds on $h_{\le0}$, $h_1$, $h_2$, $\ldots$ from \eqref{eq:standardLP}.

\begin{proposition}\label{prop:positive}
Let $K:\R^n\to\R^m$ be compactly supported, mean zero, and component-wise in $\textup{L}^\infty(\R^n)\cap \textup{BV}(\R^n)$. Then the maximal operator \eqref{eq:TK_star2} satisfies
\begin{equation}\label{eq:TKstar_bound}
\|T^\star_Kf\|_{\textup{L}^{\infty}(\R^n)} \lesssim_K \|f\|_{\calB}
\end{equation}
for every $f\in B$.
\end{proposition}

\begin{proof}
Take $f\in B$ and $r>0$. Write
\[ f = \PP_{\!\le0}f+\sum_{j=1}^{\infty}\PP_j f, \]
so that
\begin{equation}\label{eq:prop1}
K_r \ast f = K_r\ast \PP_{\!\le0}f + \sum_{j=1}^{\infty} K_r \ast \PP_j f,
\end{equation}
both with convergence in $\textup{L}^2(\R^n)$.

For the low-frequency part, we write
\begin{equation}\label{eq:prop2}
\|K_r\ast \PP_{\!\le0}f\|_{\textup{L}^\infty(\R^n;\R^m)}
\le \|K_r\|_{\textup{L}^1(\R^n;\R^m)} \|\PP_{\!\le0}f\|_{\textup{L}^\infty(\R^n)} \lesssim \|f\|_{\calB}.
\end{equation}
For the high-frequency part associated with $j\ge1$, we have
\[ K_r\ast \PP_j f \stackrel{\eqref{eq:DDisD}}{=} (K_r\ast \widetilde h_j)\ast \PP_j f, \]
so estimate \eqref{eq:general-kernel-estimate} gives
\begin{align}
\|K_r\ast \PP_j f\|_{\textup{L}^\infty(\R^n;\R^m)}
& \le \|K_r\ast \widetilde h_j\|_{\textup{L}^1(\R^n;\R^m)} \|\PP_j f\|_{\textup{L}^\infty(\R^n)} \nonumber \\
& \!\!\stackrel{\eqref{eq:general-kernel-estimate}}{\lesssim} \min\{2^jr,(2^jr)^{-1}\} \|f\|_{\calB}. \label{eq:prop3}
\end{align}
If $k\in\Z$ is such that $2^k r\le1<2^{k+1}r$, then
\[ \sum_{j=1}^{\infty}\min\{2^jr,(2^jr)^{-1}\}
\le \sum_{\substack{j\in\N\\ j\le k}} 2^j r + \sum_{\substack{j\in\N\\ j>k}} (2^j r)^{-1} \lesssim 1. \]
Now we see that the series from \eqref{eq:prop1} also converges absolutely in $\textup{L}^\infty$, and \eqref{eq:prop2} and \eqref{eq:prop3} together imply
\[ \sup_{\Q\ni r>0}\|K_r\ast f\|_{\textup{L}^\infty(\R^n;\R^m)} \lesssim \|f\|_{\calB}. \]
The last countable supremum commutes with the $\textup{L}^\infty$ norm.
\end{proof}

\begin{proof}[Proof of the boundedness part of Theorem~\ref{thm:main}]
Estimate \eqref{eq:TKstar_bound} specialized to $T^\star_K=\textup{M}^{\diamond}$ and combined with \eqref{eq:eemb} gives
\begin{equation}\label{eq:pos1bd}
\|\textup{M}^{\diamond}f\|_{\calB}
\lesssim \|\textup{M}^{\diamond}f\|_{\textup{L}^{\infty}(\R^n)} \lesssim \|f\|_{\calB},
\end{equation}
while pointwise domination by the Hardy--Littlewood maximal function \eqref{eq:domination-by-HL} ensures
\begin{equation}\label{eq:pos2bd}
\|\textup{M}^{\diamond}f\|_{\textup{L}^2(\R^n)} \lesssim \|\textup{M}f\|_{\textup{L}^2(\R^n)} \lesssim \|f\|_{\textup{L}^2(\R^n)}.
\end{equation}
Combining \eqref{eq:pos1bd} and \eqref{eq:pos2bd} we finally get $\|\textup{M}^{\diamond}f\|_B \lesssim \|f\|_B$.
\end{proof}


\section{Unboundedness of the sharp maximal operator}
\label{sec:negative}
We now prove the part of Theorem~\ref{thm:main} concerning $\textup{M}^{\sharp}$. The construction has three ingredients: a lacunary trigonometric sum whose scalar oscillation is of size $N^{1/2}$, a fixed cutoff that keeps the Besov norm uniformly bounded, and finally a dyadic dilation that makes the $\textup{L}^2$ norm harmless while losing only a logarithm in the Besov-part lower bound. The use of lacunary trigonometric sums in a similar context is unsurprising; see, for example, \cite{Kahane1985}.

Choose a real-valued function $\psi\in \textup{C}_{c}^{\infty}(\R^{n})$ that is identically $1$ on the ball $\Ball{2}{0}$. All constants below are also allowed to depend on $\psi$.
For $N\ge 1$ define a $1$-periodic trigonometric polynomial $S_N$,
\[ S_{N}(t) := \sum_{k=1}^{N} \sin(2\pi 2^k t), \]
and a function $F_N\colon\R^n \to\R$ by the formula
\[ F_N(x_1,x_2,\ldots,x_n) := \psi(x_1,x_2,\ldots,x_n) S_N(x_1). \]

\begin{lemma}\label{lm:FN_Besov}
We have
\begin{equation}\label{eq:FN_Besov2U}
\|F_N\|_{\calB} \lesssim 1
\end{equation}
uniformly in $N\ge 1$.
\end{lemma}

\begin{proof}
Let us first prove some Littlewood--Paley estimates for
\[ g_{\lambda}(x_1,x_2,\ldots,x_n) := \psi(x_1,x_2,\ldots,x_n) e^{2\pi\mathbbm{i}\lambda x_{1}}, \]
where $\lambda$ is a fixed parameter such that $|\lambda|\ge1$.
We claim that
\begin{equation}\label{eq:g_rzu}
\|\PP_j g_{\lambda}\|_{\textup{L}^\infty(\R^n)} \lesssim
\begin{cases}
|\lambda| 2^{-j}, & \text{when } |\lambda|<2^{j+2},\\
2^{jn}|\lambda|^{-n}, & \text{when } |\lambda|\ge 2^{j+2},
\end{cases}
\end{equation}
for every $j\ge1$.
Indeed, the first estimate actually holds for all $j$ and $\lambda$, because $h_j$ has mean zero, so
\[ (\PP_j g_{\lambda})(x) \stackrel{\eqref{eq:standardLP}}{=} \int_{\R^n} h_j(y)\bigl(g_{\lambda}(x-y)-g_{\lambda}(x)\bigr)\dd y \]
and then, by $|\nabla g_{\lambda}(x)|\lesssim |\lambda|$,
\[ \|\PP_j g_{\lambda}\|_{\textup{L}^{\infty}(\R^n)} \le \int_{\R^n} |h_j(y)| |y| \|\nabla g_{\lambda}\|_{\textup{L}^\infty(\R^n)}\dd y
\stackrel{\eqref{eq:standardLP}}{\lesssim} |\lambda| 2^{-j}. \]
In the second case, observe that
\[ \widehat{g_{\lambda}}(\xi_1,\xi_2,\ldots,\xi_n)=\widehat\psi(\xi_1-\lambda,\xi_2,\ldots,\xi_n). \]
On the support of $\widehat{h_j}$ one has $|\xi_1-\lambda|\ge|\lambda|/2$, while the support itself has volume $O(2^{jn})$.
Since $\widehat{\psi}$ is rapidly decreasing,
\[ \|\PP_j g_{\lambda}\|_{\textup{L}^\infty(\R^n)}
\le \|\widehat{\PP_j g_{\lambda}}\|_{\textup{L}^1(\R^n)}
\lesssim 2^{jn}|\lambda|^{-n}. \]
This proves \eqref{eq:g_rzu}.
Since no cancellation of $h_j$ was used, exactly the same Fourier argument applies to $h_{\le0}$ and gives
\begin{equation}\label{eq:g_rzu2}
\|\PP_{\!\le0}g_{\lambda}\|_{\textup{L}^\infty(\mathbb{R}^{n})}\lesssim |\lambda|^{-1},
\end{equation}
at least when $|\lambda|\ge4$, while the estimate is trivial for $1\le|\lambda|<4$.

Now consider
\[ F_N = \frac{1}{2\mathbbm{i}}\sum_{k=1}^N g_{2^k} - \frac{1}{2\mathbbm{i}}\sum_{k=1}^N g_{-2^k}. \]
For a fixed $j\ge1$,
\[ \|\PP_j F_N\|_{\textup{L}^\infty(\R^n)}
\stackrel{\eqref{eq:g_rzu}}{\lesssim} \sum_{k=1}^{j+1} 2^{k-j} + \sum_{k=j+2}^N 2^{(j-k)n} \lesssim1. \]
Also,
\[ \|\PP_{\!\le0}F_N\|_{\textup{L}^\infty(\R^n)} \stackrel{\eqref{eq:g_rzu2}}{\lesssim} \sum_{k=1}^N 2^{-k} \lesssim 1. \]
Combining these estimates proves \eqref{eq:FN_Besov2U}.
\end{proof}

\begin{lemma}\label{lm:Msharp_large}
The estimate
\begin{equation}\label{eq:Msharp_large1}
\|\textup{M}^\sharp F_N\|_{\calB} \gtrsim N^{1/2}
\end{equation}
holds with a constant independent of $N\ge 1$.
\end{lemma}

\begin{proof}
We first claim that
\begin{equation}\label{eq:Msharp_large2}
(\textup{M}^{\sharp}F_N)(x)\gtrsim N^{1/2}
\quad\text{for all }x\in\Ball{1}{0}.
\end{equation}
Fix $x=(x_1,\ldots,x_n)\in\Ball{1}{0}$ and let $a_x$ be the average of $F_N$ on $\Ball{1}{x}$. Then, by the choice of $\psi$,
\[ (\textup{M}^\sharp F_N)(x_1,x_2,\ldots,x_n)
\ge \fint_{\Ball{1}{x}} |F_N(y)-a_x| \dd y
\gtrsim \int_{x_1-1/2}^{x_1+1/2} |S_N(t)-a_x| \dd t. \]
The last quantity is bounded from below by a constant times $N^{1/2}$ as follows.
By the properties of lacunary Fourier series \cite[Chapter~5,~Theorem~8.20]{Zygmund2002},
\[ \int_{0}^{1} |S_{N}(t)| \dd t\sim N^{1/2}. \]
Also, by the fact that $S_N$ is odd and $1$-periodic,
\begin{align*}
\int_{-1/2}^{1/2} |S_N(t)-a_x| \dd t
& = \int_0^{1/2} \bigl( \big|S_N(t)-a_x\big| + \big|-S_N(t)-a_x\big| \bigr) \dd t \\
& \ge 2\int_0^{1/2} |S_N(t)| \dd t = \int_0^1 |S_N(t)| \dd t.
\end{align*}
These prove \eqref{eq:Msharp_large2}.

Next, choose a nonnegative $\rho\in \textup{C}_c^\infty(\Ball{1}{0})$ with $\int_{\R^n}\rho=1$. We claim that we always have
\begin{equation}\label{eq:just_the_rho}
\|\rho \ast g\|_{\textup{L}^{\infty}(\R^n)} \lesssim_\rho \|g\|_{\calB}
\end{equation}
for $g\in B$. For the simple proof one uses the Littlewood--Paley decomposition of $g$ and \eqref{eq:DDisD} to write, as in Proposition~\ref{prop:positive},
\[ \rho\ast g = \rho\ast \PP_{\!\le0}g + \sum_{j=1}^\infty (\rho\ast\widetilde{h}_j) \ast \PP_j g. \]
The low-frequency term is bounded in $\|\cdot\|_{\textup{L}^\infty(\R^n)}$ by $\|\rho\|_{\textup{L}^1(\R^n)} \|\PP_{\!\le0}g\|_{\textup{L}^{\infty}(\R^n)}$. Also observe
\[ \|\rho\ast\widetilde{h}_j\|_{\textup{L}^1(\R^n)}
\le \|\nabla\rho\|_{\textup{L}^1(\R^n)} \int_{\R^n} |y|\,|\widetilde{h}_j(y)| \dd y \stackrel{\eqref{eq:htildeLP}}{\lesssim_\rho} 2^{-j}, \]
because $\widetilde{h}_j$ has mean $0$.
The resulting geometric sum gives the desired estimate.

Now, estimate \eqref{eq:Msharp_large2} gives
\[ (\rho\ast \textup{M}^\sharp F_N)(0) = \int_{\Ball{1}{0}} \rho(y)(\textup{M}^\sharp F_N)(-y)\dd y \gtrsim N^{1/2}. \]
Applying \eqref{eq:just_the_rho} with $g=\textup{M}^\sharp F_N$ we then furthermore get
\[  N^{1/2}\lesssim \|\rho\ast \textup{M}^\sharp F_N\|_{\textup{L}^\infty(\R^n)} \stackrel{\eqref{eq:just_the_rho}}{\lesssim_\rho} \|\textup{M}^\sharp F_N\|_{\calB}, \]
which is \eqref{eq:Msharp_large1}.
\end{proof}

Our actual counterexample will be a sequence of appropriate dilates of the functions $F_N$.
For an integer $m$, write $\Dil_m f$ for the $\textup{L}^\infty$-normalized dyadic dilate of $f$,
\[ (\Dil_m f)(x):=f(2^m x). \]

\begin{lemma}\label{lm:dilations}
For every integer $m\ge 0$ we have
\begin{equation}\label{eq:Dm-positive-Besov}
\|\Dil_m f\|_{\calB}\lesssim \|f\|_{\calB},
\end{equation}
and
\begin{equation}\label{eq:Dm-negative-Besov}
\|\Dil_{-m} f\|_{\calB}\lesssim (m+1)\|f\|_{\calB}.
\end{equation}
\end{lemma}

\begin{proof}
For the purpose of this proof extend the definitions
\begin{align*}
\widehat{h_{\le j}}(\xi) & := \varphi\Bigl(\frac{|\xi|}{2^j}\Bigr), \quad \xi\in\R^n,\\
\widehat{h_j}(\xi) & := \varphi\Bigl(\frac{|\xi|}{2^j}\Bigr) - \varphi\Bigl(\frac{|\xi|}{2^{j-1}}\Bigr), \quad \xi\in\R^n,\\
\PP_{\!\le j} f & := f \ast h_{\le j},\\
\PP_j f & := f \ast h_j
\end{align*}
to every $j\in\Z$. The operators $\PP_{\!\le j}$ and $\PP_j$ are clearly uniformly bounded on $\textup{L}^\infty(\R^n)$.

The case $m=0$ is trivial, so assume $m\ge1$.
We first prove \eqref{eq:Dm-positive-Besov} using the identities
\begin{align*}
\PP_j \Dil_m f & =
\begin{cases}
\Dil_m\PP_{j-m} \PP_{\!\le0} f & \text{for } 1\le j\le m-1,\\
\Dil_m\PP_{j-m}(\PP_{\!\le0}+\PP_1)f & \text{for } j=m,\\
\Dil_m\PP_{j-m}f & \text{for } j\ge m+1,
\end{cases} \\
\PP_{\!\le0}\Dil_m f & = \Dil_m\PP_{\!\le -m} \PP_{\!\le0} f,
\end{align*}
which are respective consequences of the following immediate equalities on the Fourier side:
\begin{align*}
& \widehat{h_j}(\xi) \frac{1}{2^{mn}} \widehat{f}\Bigl(\frac{\xi}{2^m}\Bigr)
= \frac{1}{2^{mn}} \widehat{f}\Bigl(\frac{\xi}{2^m}\Bigr) \widehat{h_{j-m}}\Bigl(\frac{\xi}{2^m}\Bigr) \\
& = \frac{1}{2^{mn}} \widehat{f}\Bigl(\frac{\xi}{2^m}\Bigr) \widehat{h_{j-m}}\Bigl(\frac{\xi}{2^m}\Bigr) \times
\begin{cases}
\widehat{h_{\le0}}(\xi/2^m) & \text{\!\!for } 1\le j\le m-1,\\
(\widehat{h_{\le0}}(\xi/2^m)+\widehat{h_{1}}(\xi/2^m)) & \text{\!\!for } j=m,\\
1 & \text{\!\!for } j\ge m+1,
\end{cases} 
\end{align*}
\begin{align*}
\widehat{h_{\le0}}(\xi) \frac{1}{2^{mn}} \widehat{f}\Bigl(\frac{\xi}{2^m}\Bigr)
& = \frac{1}{2^{mn}} \widehat{h_{\le -m}}\Bigl(\frac{\xi}{2^m}\Bigr) \widehat{f}\Bigl(\frac{\xi}{2^m}\Bigr) \\
& = \frac{1}{2^{mn}} \widehat{h_{\le -m}}\Bigl(\frac{\xi}{2^m}\Bigr) \widehat{h_{\le 0}}\Bigl(\frac{\xi}{2^m}\Bigr) \widehat{f}\Bigl(\frac{\xi}{2^m}\Bigr).
\end{align*}
The aforementioned identities give
\begin{align*}
\|\PP_j \Dil_m f\|_{\textup{L}^{\infty}(\R^n)} & \lesssim
\begin{cases}
\|\PP_{\!\le0} f\|_{\textup{L}^{\infty}(\R^n)} & \text{for } 1\le j\le m-1,\\
\|\PP_{\!\le0}f\|_{\textup{L}^{\infty}(\R^n)} + \|\PP_1 f\|_{\textup{L}^{\infty}(\R^n)} & \text{for } j=m,\\
\|\PP_{j-m}f\|_{\textup{L}^{\infty}(\R^n)} & \text{for } j\ge m+1,
\end{cases} \\
\|\PP_{\!\le0}\Dil_m f\|_{\textup{L}^{\infty}(\R^n)} & \lesssim \|\PP_{\!\le0} f\|_{\textup{L}^{\infty}(\R^n)},
\end{align*}
which implies \eqref{eq:Dm-positive-Besov}.

Estimate \eqref{eq:Dm-negative-Besov} is even more direct.
This time we have the identities
\begin{align*}
\PP_j\Dil_{-m}f & = \Dil_{-m}\PP_{j+m}f \quad \text{for } j\ge1, \\
\PP_{\!\le0}\Dil_{-m}f & = \Dil_{-m}\Bigl(\PP_{\!\le0}f+\sum_{j=1}^{m}\PP_j f\Bigr),
\end{align*}
which follow from
\begin{align*}
\widehat{h_{j}}(\xi) 2^{mn} \widehat{f}(2^m\xi) & = 2^{mn} \widehat{h_{j+m}}(2^m\xi) \widehat{f}(2^m\xi) \quad \text{for } j\ge1, \\
\widehat{h_{\le0}}(\xi) 2^{mn} \widehat{f}(2^m\xi) & = 2^{mn} \widehat{h_{\le m}}(2^m\xi) \widehat{f}(2^m\xi) \\[-2mm]
& = 2^{mn} \Bigl( \widehat{h_{\le 0}}(2^m\xi) + \sum_{j=1}^m \widehat{h_{j}}(2^m\xi) \Bigr) \widehat{f}(2^m\xi)
\end{align*}
and imply
\begin{align*}
\|\PP_j\Dil_{-m}f\|_{\textup{L}^{\infty}(\R^n)} & \lesssim \|\PP_{j+m}f\|_{\textup{L}^{\infty}(\R^n)} \quad \text{for } j\ge1, \\
\|\PP_{\!\le0}\Dil_{-m}f\|_{\textup{L}^{\infty}(\R^n)} & \lesssim \Bigl(\|\PP_{\!\le0}f\|_{\textup{L}^{\infty}(\R^n)} + \sum_{j=1}^{m} \|\PP_j f\|_{\textup{L}^{\infty}(\R^n)} \Bigr).
\end{align*}
These establish \eqref{eq:Dm-negative-Besov}.
\end{proof}

Now we are ready to disprove the boundedness of $\textup{M}^\sharp$.

\begin{proof}[Proof of the unboundedness part of Theorem~\ref{thm:main}]
The sharp maximal operator commutes with dyadic dilations $\Dil_m$,
\begin{equation}\label{eq:Msharp-dilation-commute}
\textup{M}^\sharp \Dil_m f = \Dil_m \textup{M}^\sharp f
\end{equation}
for $m\in\Z$. Indeed, after the change of variables $y\mapsto 2^{-m}y$, averages over balls $\Ball{r}{x}$ are transformed into averages over balls $\Ball{2^m r}{2^m x}$.

Choose $m_N\in\mathbb N\cup\{0\}$ by $m_N := \lfloor \log_2 N \rfloor$ and define
\[ f_N=\Dil_{m_N}F_N. \]
By
\[ \|F_N\|_{\textup{L}^2(\R^n)} \lesssim \Bigl(\int_{0}^{1} |S_{N}(t)|^2 \dd t\Bigr)^{1/2} \lesssim N^{1/2}, \]
Lemma~\ref{lm:dilations}, and Lemma~\ref{lm:FN_Besov}, we have
\[ \|f_N\|_{\textup{L}^2(\R^n)}=2^{-m_Nn/2}\|F_N\|_{\textup{L}^2(\R^n)} \lesssim N^{-n/2+1/2} \lesssim 1 \]
and
\[ \|f_N\|_{\calB}\lesssim\|F_N\|_{\calB} \stackrel{\eqref{eq:FN_Besov2U}}{\lesssim} 1. \]
Consequently,
\[ \sup_{N\ge1}\|f_N\|_{B}<\infty. \]
On the other hand, using Lemma~\ref{lm:Msharp_large}, dilation invariance, and Lemma~\ref{lm:dilations}, we obtain
\[ N^{1/2} \stackrel{\eqref{eq:Msharp_large1}}{\lesssim} \|\textup{M}^{\sharp}F_N\|_{\calB}
\stackrel{\eqref{eq:Msharp-dilation-commute}}{=} \|\Dil_{-m_N}\textup{M}^{\sharp}f_N\|_{\calB}
\stackrel{\eqref{eq:Dm-negative-Besov}}{\lesssim} (m_N+1)\|\textup{M}^{\sharp}f_N\|_{\calB}. \]
Therefore
\[ \|\textup{M}^{\sharp}f_N\|_{B}
\ge \|\textup{M}^{\sharp}f_N\|_{\calB}
\gtrsim \frac{N^{1/2}}{\log_2 N + 1}\rightarrow\infty \]
as $N\to\infty$.
\end{proof}

Finally, we use the common ``amplification trick'' to glue members of an ``asymptotically bad'' sequence $(f_N)_{N=1}^\infty$ into a single function $f$ which is then ``universally bad.''

\begin{proposition}\label{prop:nonself}
There exists $f\in B$ such that $\textup{M}^{\sharp}f\notin B$.
\end{proposition}

\begin{proof}
Fix a nonnegative $\rho\in\textup{C}_c^\infty(\Ball{1}{0})$ with $\int_{\R^n}\rho=1$, and, for an integer $m\ge0$, put $\rho_m(x):=2^{mn}\rho(2^m x)$.
The identity
\[ (\rho_m\ast g)(2^{-m}z) = (\rho\ast\Dil_{-m}g)(z) \]
and \eqref{eq:just_the_rho}, followed by \eqref{eq:Dm-negative-Besov}, give
\begin{equation}\label{eq:scaled-rho}
\|\rho_m\ast g\|_{\textup{L}^\infty(\R^n)} \lesssim_\rho (m+1)\|g\|_{\calB}
\end{equation}
for $g\in B$.

Choose $R\ge2$ with $\operatorname{supp}\psi\subset\Ball{R}{0}$ and points 
\[ x_k := k (4R+5) (1,0,\ldots,0). \]
For $k\ge1$, set $N_k:=2^{4k}$, $m_k:=4k$, $a_k:=2^{-k}$, and
\[ u_k(x):=F_{N_k}\bigl(2^{m_k}(x-x_k)\bigr). \]
Since $m_k=\lfloor\log_2 N_k\rfloor$, the functions $u_k$ are translates of the functions $f_{N_k}$ from the preceding proof.
Consequently, $\sup_k\|u_k\|_B<\infty$, and the series
\[ f:=\sum_{k=1}^\infty a_k u_k \]
converges in $B$ by completeness. It also defines a locally finite pointwise sum, since $\operatorname{supp}u_k\subset\Ball{R2^{-m_k}}{x_k}$. The pointwise sum represents the same element of $B$, by convergence in $\textup{L}^2(\R^n)$.

Let $r_k:=2^{-m_k}$ and $|x-x_k|<r_k$. The ball $\Ball{r_k}{x}$ meets no support other than that of $u_k$, and after the change of variables $z=2^{m_k}(y-x_k)$ the cutoff $\psi$ is identically one on the corresponding unit ball. The unit-radius calculation in the proof of \eqref{eq:Msharp_large2} therefore gives
\[ (\textup{M}^{\sharp}f)(x) \ge \fint_{\Ball{r_k}{x}} \Big|f(y)-\fint_{\Ball{r_k}{x}}f\Big|\dd y \gtrsim a_k N_k^{1/2}. \]
All constants are independent of $k$.

Suppose that $g:=\textup{M}^{\sharp}f$ belongs to $B$. Since $g\ge0$ and $\rho_{m_k}$ is supported in $\Ball{r_k}{0}$ and has integral one, the preceding lower bound and \eqref{eq:scaled-rho} imply
\[ 2^k=a_kN_k^{1/2} \lesssim (\rho_{m_k}\ast g)(x_k) \lesssim_\rho (4k+1)\|g\|_{\calB}. \]
The convolutions are continuous, so their essential supremum bounds apply at the indicated points. Letting $k\to\infty$ gives a contradiction, proving that $\textup{M}^{\sharp}f\notin B$.
\end{proof}


\subsection*{Declaration of AI usage}
OpenAI's \emph{ChatGPT} 5.4 Pro and \emph{ChatGPT} 5.5 Pro were used to investigate useful examples of Banach spaces,
to the extent specified toward the end of the introduction. \emph{Overleaf AI assistant} was used to suggest additional references. The ideas, the results, the proofs, the bibliography, and the writing of the manuscript are entirely the author's work.


\section*{Acknowledgments and funding}
The author is grateful to Aleksandar Bulj for useful discussions.
The author would also like to thank the anonymous referees for several suggestions on the exposition.

This work was supported in part by the Croatian Science Foundation under the project HRZZ-IP-2022-10-5116 (\emph{FANAP}).
This paper was also supported in part by the European Union -- NextGenerationEU through the National Recovery and Resilience Plan 2021--2026, via an institutional grant of the University of Zagreb Faculty of Science, IK IA 1.1.3, \emph{Impact4Math}.


\bibliographystyle{plainurl}
\bibliography{mazya_maximal_operators}{}

\end{document}